\documentclass[12pt]{amsart}
\usepackage{amstext}
\usepackage{amssymb,latexsym}
\usepackage{graphicx}
\def \mod {\text{ mod }}
\def \diam {\text{diam}}

\newtheorem{thm}{Theorem}[section]

\newtheorem{rmk}[thm]{Remark}

\newtheorem{defin}[thm]{Definition}
\newtheorem{lemma}[thm]{Lemma}

\begin{document}

\title{Radio numbers for generalized prism graphs}
        \date{\today}
   
\author{Paul Martinez}
\address{\hskip-\parindent
 Paul Martinez\\
 California State University Channel Islands.}
\email{paul.martinez@csuci.edu}

\author{Juan Ortiz}
\address{\hskip-\parindent
 Juan Ortiz\\
Lehigh University.}
\email{jpo208@lehigh.edu}

\author{Maggy Tomova}
\address{\hskip-\parindent
 Maggy Tomova\\
The University of Iowa.}
\email{mtomova@math.uiowa.edu}

\author{Cindy Wyels}
\address{\hskip-\parindent
 Cindy Wyels\\
  California State University Channel Islands.}
\email{cynthia.wyels@csuci.edu}


\thanks{This research was initiated under the auspices of an MAA (SUMMA) Research Experience for Undergraduates program funded by NSA, NSF, and Moody's, and hosted at CSU Channel Islands during Summer, 2006. We are grateful to all for the opportunities provided.}

\begin{abstract}
 A radio labeling is an assignment $c:V(G)
\rightarrow \textbf{N}$ such that every distinct pair of vertices $u,v$ satisfies the inequality
$d(u,v)+|c(u)-c(v)|\geq \diam(G)+1$. The span of a radio labeling is the maximum value. The radio number of
$G$, $rn(G)$, is the minimum span over all radio labelings of $G$.
Generalized prism graphs, denoted $Z_{n,s}$, $s \geq 1$, $n\geq s$, have vertex set $\{(i,j)\,|\, i=1,2 \text{ and } j=1,...,n\}$ and edge set $\{((i,j),(i,j \pm 1))\} \cup \{((1,i),(2,i+\sigma))\,|\,\sigma=-\left\lfloor\frac{s-1}{2}\right\rfloor\,\ldots,0,\ldots,\left\lfloor\frac{s}{2}\right\rfloor\}$. In this
   paper we determine the radio number of $Z_{n,s}$ for $s=1,2$ and 3. In the process we
develop techniques that are likely to be of use in determining radio numbers of other families of graphs.

\vspace{.1in}\noindent \textbf{2000 AMS Subject Classification:} 05C78 (05C15)
\end{abstract}

\maketitle

\section{Introduction}
Radio labeling is a graph labeling problem, suggested by Chartrand, et al \cite{CEHZ}, that is analogous to assigning frequencies to FM channel stations so as to avoid signal interference. Radio stations that are close geographically must have frequencies that are very different, while radio stations with large geographical separation may have similar frequencies. Radio labeling for a number of families of graphs has been studied, for example see \cite{hypercube,m-ary trees,trees,DaphneSquare, square paths,multilevel,cycles}. A survey of known results about radio labeling can be found in \cite{survey}. In this paper we determine the radio number of certain generalized prism graphs.

All graphs we consider are simple and connected. We denote by $V(G)$ the vertices of $G$. We use $d_G(u,v)$ for the length of the shortest path in $G$ between $u$ and $v$. The diameter of $G$, $\diam(G)$, is the maximum distance in $G$. A {\em radio labeling} of $G$ is a function $c_G$ that assigns to each vertex $u \in V(G)$ a positive integer $c_G(u)$ such that any two distinct vertices $u$ and $v$ of $G$ satisfy the radio condition:
\[d_G(u,v)+|c_G(u)-c_G(v)|\geq \diam(G)+1.\]
The {\em span} of a radio labeling is the maximum value of $c_G$. Whenever $G$ is clear from context, we simply write $c(u)$ and $d(u,v)$. The {\em radio number} of $G$, $rn(G)$, is the minimum span over all possible radio labelings of $G$\footnote{We use the convention that $\textbf{N}$ consists of the positive integers. Some authors let $\textbf{N}$ include $0$, with the result that radio numbers using this definition are one less than radio numbers determined using the positive integers.}.

In this paper we determine the radio number of a family of graphs that consist of two $n$-cycles together with some edges connecting vertices from different cycles. The motivating example for this family of graphs is the prism graph, $Z_{n,1}$, which is the Cartesian product of $P_2$, the path on $2$ vertices, and $C_n$, the cycle on $n$ vertices. In other words, a prism graph consists of $2$ $n$-cycles with vertices labeled $(1, i), i=1,\ldots,n$ and $(2, i), i=1,\ldots,n$ respectively together with all edges between pairs of vertices of the form $(1,i)$ and $(2,i)$. Generalized prism graphs, denoted $Z_{n,s}$ have the same vertex set as prism graphs but have additional edges. In particular, vertex $(1,i)$ is also adjacent to
   $(2,i+\sigma)$ for each $\sigma$ in $\{-\left\lfloor\frac{s-1}{2}\right\rfloor\,\ldots,0,\ldots,\left\lfloor\frac{s}{2}\right\rfloor\}$, see Definition \ref{def:gpg}. 
   
\medskip
\noindent\textbf{Main Theorem:} {\em Let $Z_{n,s}$ be a generalized prism graph with $1\leq s \leq 3$, and $(n,s)
\neq (4,3)$. Let $n=4k+r$, where $k \geq 1$, and $r=0,1,2,3$. Then \[rn(Z_{n,s})= (n-1)\phi(n,s)+2.\] where $\phi(n,s)$ is given in the following table:}
\begin{center}
$\phi(n,s):\quad$
\begin{tabular}{c|c|c|c}

{} & {\bf \textit{s}=1} & {\bf \textit{s}=2} & {\bf \textit{s}=3} \\ \hline {\bf \textit{r}=0}& {$k+2$} & $k+1$ &
{$k+2$} \\ \hline {\bf \textit{r}=1} & {$k+2$} & $k+2$ & {$k+1$} \\ \hline {\bf \textit{r}=2} & {$k+3$} & $k+2$
&{$k+2$} \\ \hline {\bf \textit{r}=3}& {$k+2$} & $k+3$ & {$k+2$} \\
\end{tabular}
\end{center}
{\em In addition, $rn(Z_{3,3})=6$ and $rn(Z_{4,3})=9$.}

\section{Preliminaries}

We will use pair notation to identify the vertices of the graphs with the first coordinate identifying the cycle, $1$ or $2$, and the second coordinate identifying the position of the vertex within the cycle, $1,...,n$. To avoid complicated notation, identifying a vertex as $(i,j)$ will always imply that the first coordinate is taken modulo 2 with $i \in \{1,2\}$ and the second coordinate is taken modulo $n$ with $j \in \{1,...,n\}$.

\begin{defin} \label{def:gpg}A generalized prism graph, denoted $Z_{n,s}$, $s \geq 1$, $n\geq s$, has vertex set $\{(i,j)\,|\, i=1,2 \textrm{ and } j=1,...,n\}$. Vertex $(i,j)$ is adjacent to  $(i,j \pm 1)$. In addition, $(1,i)$ is adjacent to
   $(2,i+\sigma)$ for each $\sigma$ in $\{-\left\lfloor\frac{s-1}{2}\right\rfloor\,\ldots,0,\ldots,\left\lfloor\frac{s}{2}\right\rfloor\}$.

   The two $n$-cycle subgraphs of $Z_{n,s}$ induced by 1) all vertices of the form $(1,j)$ and 2) all vertices of the form $(2,j)$ are called {\em principal cycles}.

   \end{defin}
In this notation, the prism graphs $C_n \square P_2$ are $Z_{n,1}$. We note that $Z_{n,2}$ graphs are isomorphic to
the squares of even cycles, $C_{2n} ^2$, whose radio number is determined in \cite{DaphneSquare}. The graphs
$Z_{8,1}$, $Z_{8,2}$, and $Z_{8,3}$ are illustrated in Figure \ref{fig:3graphpic}.

\begin{figure}
\begin{center} \includegraphics[scale=.35]{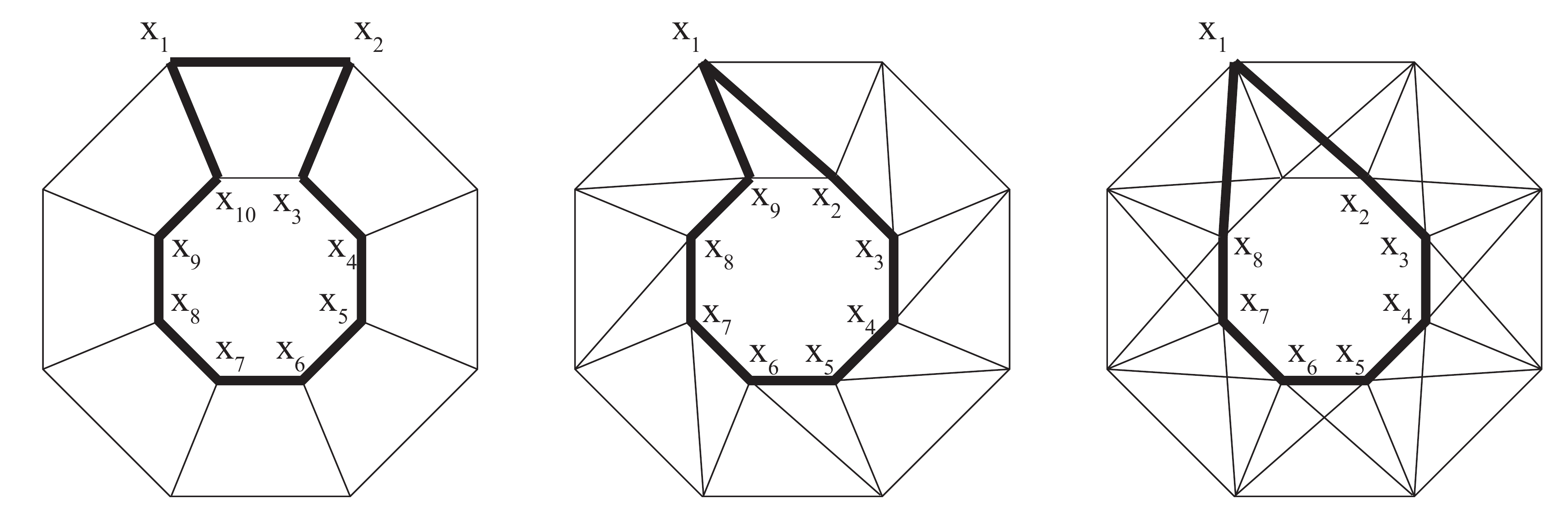}
\end{center}
\caption{$Z_{8,1}$, $Z_{8,2}$, and $Z_{8,3}$, with standard cycles depicted by thickened edges} \label{fig:3graphpic} \end{figure}

\begin{rmk} \label{rmk:radius} Note that $diam (Z_{n,s})=\left\lfloor\frac{n+3-s}{2} \right\rfloor$ for $s=1,2,3$.
\end{rmk}

Our general approach to determining the radio number of $Z_{n,s}$ consists of two steps. We first establish a lower
bound for the radio number. Suppose $c$ is a radio labeling of the vertices of $Z_{n,s}$. We can rename the vertices
of $Z_{n,s}$ with $\{\alpha_i\,|\,i=1,...,2n\}$, so that $c(\alpha_i)<c(\alpha_j)$ whenever $i<j$. We determine the
minimum label difference between $c(\alpha_i)$ and $c(\alpha_{i+2})$, denoted $\phi(n,s)$, and use it to establish
that $rn(Z_{n,s}) \geq 2+(n-1)\phi(n,s)$. We then demonstrate an algorithm that establishes that this lower bound is
in fact the radio number of the graph. We do this by defining a {\em position function} $p:V(G) \rightarrow \{\alpha_i
\,|\, i=1,...,2n\}$ and a {\em labeling function} $c:\{\alpha_i\} \rightarrow Z_+$ that has span $(n-1)\phi(n,s)+2$. We
prove that $p$ is a bijection, i.e., every vertex is labeled exactly once, and that all pairs of vertices together with the labeling $c \circ p^{-1}$ satisfy the
radio condition.

Some small cases of generalized prism graphs with $s=3$ do not follow the general pattern, so we discuss these first. First note that
$Z_{3,3}$ has diameter $1$ and thus can be radio-labeled using consecutive integers, i.e., $rn(Z_{3,3})=6$. To
determine $rn(Z_{4,3})$, note that the diameter of $Z_{4,3}$ is 2. Therefore the radio number of $Z_{4,3}$ is the same as the $L(2,1)$-number of the graph as defined in \cite{L21}. This prism graph is isomorphic to the join of two copies of $K_2\cup K_2$ where $K_2 \cup K_2$ is the disconnected graph with two components each with 2 vertices and one edge. By  \cite{L21}, it follows that $rn(Z_{4,3})=9$. (We thank the referee for pointing out this proof.)

To simplify many of the computations that follow, we make use of the existence of certain special cycles in the
graphs.

\begin{defin} Suppose a graph $G$ contains a subgraph $H$ isomorphic to
a cycle, and let $v \in V(H)$.
\begin{itemize}
\item We will call $H$ a $v$-tight cycle if for every $u\in V(H)$, $d_G(u,v)=d_H(u,v)$.
\item We will call $H$ a tight cycle if for every pair of vertices $u,w$ in $H$, $d_G(u,w)=d_H(u,w)$.
\end{itemize}
\end{defin}
\noindent We note that $H$ is a tight cycle if and only if $H$ is $v$-tight for every $v \in V(H)$.

\begin{rmk}\label{rmk:principletightcycles}
Each of the two principal $n$-cycles is tight.
\end{rmk}

Particular $v$-tight cycles of maximum length play an important role in our proofs. Figure \ref{fig:3graphpic} uses bold edges to indicate a particular $(1,1)$-tight cycle of maximum length for each of the three types of graphs. The figure illustrates these cycles in the particular case when $n=8$ but it is easy to generalize the construction to any $n$. We will call these particular maximum-length $(1,1)$-tight cycles in $Z_{n,s}$ {\em standard}. Thus each generalized prism graph with $1 \leq s\leq 3$ has a standard cycle. For convenience, we will use a second set of names for the vertices of a standard cycle when focusing on properties of, or distance within, the standard cycles. The vertices of a standard cycle for $Z_{n,s}$ will be labeled $X^s_i$, $i=1,..,n+3-s$, where
\[X^1_i= \left\{
  \begin{array}{ll}
   $(1,1)$, & \hbox{if $i=1$,} \\
    $(1,2)$, & \hbox{if $i=2$,}\\
    $(2,$i-1$)$   & \hbox{if $i \geq 3$.}
  \end{array}
\right.\] 
and for $k=2,3$, 
\[X^k_i= \left\{
  \begin{array}{ll}
   $(1,1)$, & \hbox{if $i=1$,} \\
    $(2,$i$)$   & \hbox{if $i \geq 2$.}
  \end{array}
\right.\]
These labels are illustrated in Figure \ref{fig:3graphpic}.

\begin{rmk} \label{rmk:standardtightcycles} The standard cycles depicted in Figure \ref{fig:3graphpic} are $(1,1)$-tight and have $n+3-s$ vertices. Therefore each standard cycle has diameter equal to the diameter of its corresponding $Z_{n,s}$ graph.

\end{rmk}

\section{Lower Bound}

Suppose $c$ is a radio labeling of $G$ with minimum span. Intuitively, building such a labeling requires one to find
groups of vertices that are pairwise far from each other so they may be assigned labels that have small pairwise
differences. The following lemma will be used to determine the maximal pairwise distance in a group of 3
vertices in $Z_{n,s}$. This leads to Lemma \ref{lem:phifunct}, in which we determine the minimum difference
between the largest and smallest label in any group of 3 vertex labels. 
\begin{lemma}  \label{lem:vexdis4AP}
Let $\{u,v,w\}$ be any subset of size 3 of $V(Z_{n,s})$, $1\leq s\leq 3$, with the exception of $\{(1,j), (2,j),(i,l)\}$ in
$V(Z_{n,3})$. Then $d(u,v)+d(v,w)+d(u,w)\leq n+3-s$.
\end{lemma}

\begin{proof} Note that if $u$,$v$, and $w$ lie on a cycle of length $t$, then $d(u,w)+d(v,u)+d(w,u) \leq t$.
If $u$,$v$, and $w$ lie on the same principal $n$-cycle, the desired result follows immediately, as all
three vertices lie on a cycle of length $n$, and $1\leq s\leq 3$.

Suppose $u$, $v$, and $w$ do not all lie on the same principal $n$-cycle. Without loss of generality, assume $u=(1,1)$,
and $v$ and $w$ lie on the second principal $n$-cycle. Then for $s=1\textrm{ or } 2$, the standard cycle includes
$(1,1)$ and all vertices $(2,i)$, so $v$ and $w$ lie on the standard cycle. For $s=3$, the standard cycle includes all
vertices $(2,i)$, $i>1$. As the triple $\{(1,j), (2,j),(i,l)\}$ in $V(Z_{n,3})$ was eliminated in the hypothesis, it
follows that for $s\leq 3$ all three vertices lie on the appropriate standard cycle. As the standard cycle in each
case is of length $n-s+3$, the result follows as above.
 \end{proof}

\begin{lemma} \label{lem:phifunct}

Let $c$ be a radio labeling of $Z_{n,s}$, $1\leq s \leq 3$ and $n=4k+r$, where $k \geq 1$, $r=0,1,2,3$, and $(n,s)
\neq (4,3)$. Suppose $V(G)=\{\alpha_i \,|\, i=1,...,2n\}$ and $c(\alpha_i)< c(\alpha_{j})$ whenever $i<j$. Then we
have $|c(\alpha_{i+2})-c(\alpha_{i})| \geq \phi(n,s)$, where the values of $\phi(n,s)$ are given in the following
table.

\begin{center}
$\phi(n,s):\quad$
\begin{tabular}{c|c|c|c}

{} & {\bf \textit{s}=1} & {\bf \textit{s}=2} & {\bf \textit{s}=3} \\ \hline {\bf \textit{r}=0}& {$k+2$} & $k+1$ &
{$k+2$} \\ \hline {\bf \textit{r}=1} & {$k+2$} & $k+2$ & {$k+1$} \\ \hline {\bf \textit{r}=2} & {$k+3$} & $k+2$
&{$k+2$} \\ \hline {\bf \textit{r}=3}& {$k+2$} & $k+3$ & {$k+2$} \\
\end{tabular}
\end{center}

\end{lemma}

\begin{proof}

First assume $\{\alpha_{i}, \alpha_{i+1}, \alpha_{i+2}\}$ are any three vertices in any generalized prism graph with
$1\leq s\leq 3$ except $\{(1,j), (2,j),(i,l)\}$ in $V(Z_{n,3})$.
 Apply the radio condition to each pair in the vertex set $\{\alpha_{i}, \alpha_{i+1},
\alpha_{i+2}\}$ and take the sum of the three inequalities. We obtain

\begin{eqnarray}
&d(\alpha_{i+1},\alpha_{i})&+d(\alpha_{i+2},\alpha_{i+1})+d(\alpha_{i+2},\alpha_{i}) \nonumber \\
&&+|c(\alpha_{i+1})-c(\alpha_i)|+|c(\alpha_{i+2})-c(\alpha_{i+1})|+|c(\alpha_{i+2})-c(\alpha_i)|\nonumber \\
&&\geq 3\,\diam(Z_{n,s})+3.
\end{eqnarray}

We drop the absolute value signs because $c(\alpha_i) < c(\alpha_{i+1}) < c(\alpha_{i+2})$, and use Lemma
\ref{lem:vexdis4AP} to rewrite the inequality as 
\[c(\alpha_{i+2})-c(\alpha_i) \geq
\frac{1}{2}\left(3+3\,\diam(Z_{n,s})-(n+3-s)\right).\]
 The table in the statement of the lemma has been generated by substituting the appropriate values for $\diam(Z_{n,s})$ from Remark \ref{rmk:radius} and simplifying. As the computations are straightforward but tedious, they are not included.

It remains to consider the case $\{\alpha_{i}, \alpha_{i+1}, \alpha_{i+2}\}=\{(1,j), (2,j),(i,l)\}$ in
$V(Z_{n,3})$. From the radio condition, it follows that
\[d\left((1,j),(2,j)\right)+|c(1,j)-c(2,j)|\geq \left\lfloor \frac{n}{2} \right\rfloor+1,\]
and so
\[|c(1,j)-c(2,j)|\geq \left\lfloor \frac{n}{2}\right\rfloor \geq 2k.\]
Thus we may conclude
\[|c(\alpha_{i+2})-c(\alpha_i)|\geq
|c(1,j)-c(2,j)| \geq 2k.\]

If $k \geq 2$, then
\[|c(\alpha_{i+2})-c(\alpha_i)|\geq 2k \geq k+2 \geq \phi(n,3).\]

If $k=1$, recall that $Z_{4,3}$ is excluded in the hypothesis. It is easy to verify that for $n=5,6,7$, $\phi(n,3)=\left\lfloor
\frac{n}{2}\right\rfloor$.
\end{proof}

\begin{rmk} \label{rmk:phi}
For all values of $n$ and $1\leq s\leq 3$, $2\phi(n,s)\geq \diam(Z_{n,s})$.
\end{rmk}

\begin{thm} \label{thm:lowerbound}
For every graph $Z_{n,s}$ with $1\leq s\leq 3$,
\[rn(Z_{n,s})\geq (n-1)\phi(n,s)+2.\]

\end{thm}
\begin{proof}
We may assume $c(\alpha_{1})=1$. By Lemma \ref{lem:phifunct}, $|c(\alpha_{i+2})-c(\alpha_{i})| \geq \phi(n,s)$, so
$c(\alpha_{2i-1})=c(\alpha_{1+2(i-1)})\geq (i-1)\phi(n,s)+1$. Note that all generalized prism graphs have $2n$
vertices. As
\[c(\alpha_{2n-1})\geq(n-1)\phi(n,s) +1,\] we have
\[c(\alpha_{2n})\geq c(\alpha_{2n-1})+1= (n-1)\phi(n,s)+2.\]

\end{proof}

\section{Upper Bound}
To construct a labeling for $Z_{n,s}$ we will define a position function $p :\{ \alpha_i \,|\, i=1,...,2n\}
\rightarrow V(Z_{n,s})$ and then a labeling function $c: \{\alpha_i \,|\, i=1,...,2n\} \rightarrow \textbf{N}$. The
composition $c\circ p^{-1}$ gives an algorithm to label $Z_{n,s}$, and this labeling has span equal to the lower bound
found in Theorem \ref{thm:lowerbound}. The labeling function depends only
on the function $\phi(n,s)$ defined in Lemma \ref{lem:phifunct}.
\begin{defin}  \label{defin:labeling}
Let $\{\alpha_i \,|\, i=1,...,2n\}$ be the vertices of $Z_{n,s}$. Define $c:\{\alpha_i \,|\, i=1,...,2n\}\rightarrow
\textbf{N}$ to be the function
\[c(\alpha_{2i-1})=1+(i-1)\phi(n,s)\textrm{, and}\]
\[c(\alpha_{2i})=2+(i-1)\phi(n,s).\]
\end{defin}

Suppose $f$ is any labeling of any graph $G$. If for some $u,v \in V(G)$ the inequality $|f(u)-f(v)| \geq \textrm{diam}(G)$ holds, then the radio condition is always satisfied for $u$ and $v$. The next lemma uses this property to limit the number of vertex pairs for which it must be checked that the labeling $c$ of Definition \ref{defin:labeling} satisfies the
radio condition.

\begin{lemma}\label{lem:fcol}
Let $\{\alpha_i \,|\, i=1,...,2n\}$ be the vertices of $Z_{n,s}$ and $c$ be the labeling of Definition
\ref{defin:labeling}. Then whenever $|l-k|\geq 4$, $d(\alpha_l,\alpha_k)+|c(\alpha_l)-c(\alpha_k)|\geq
\textrm{diam}(Z_{n,s})+1$.
\end{lemma}

\begin{proof}
Without loss of generality, let $l>k$. Since $c(\alpha_{k+4})\leq c(\alpha_l)$, it follows that
\[c(\alpha_l)-c(\alpha_k)\geq c(\alpha_{k+4})-c(\alpha_k)=2\phi(n,s).\]
From Remark \ref{rmk:phi} it follows that
\[|c(\alpha_l)-c(\alpha_k)|+d(\alpha_l,\alpha_k)\geq 2\phi(n,s)+1\geq \textrm{diam}(Z_{n,s})+1.\]
\end{proof}

We will need to consider four different position functions depending on $n$ and $s$. Each of these position functions together with the labeling function in Definition \ref{defin:labeling} gives an algorithm for labeling a particular $Z_{n,s}$.

\begin{center}
\textbf{Case 1: $n=4k+r$, $r=1,2,3$ and $s\leq 3$}

\textbf{except $n=4k+2$ when $k$ is even and $s=3$} \end{center}

The idea is to find a position function which allows pairs of consecutive integers to be used as labels as often as possible. To use consecutive integers we need to find pairs of vertices in $Z_{n,s}$ with distance equal to the diameter. We will do this by taking advantage of the standard cycles for each value of $s$.

\begin{lemma} \label{lem:diameteraway}
   For all $n \geq 3$ and $s\leq3$,  $d\left((1,y),(2, y+D)\right)=\diam(Z_{n,s})$ where 
\[D=\left\{
           \begin{array}{ll}
             \left\lfloor
   {\frac{n+1}{2}}\right\rfloor, & \hbox{for $s=1 \textrm{ and } 3$, } \\ \\
             \left\lfloor
   {\frac{n+2}{2}}\right\rfloor, & \hbox{for $s=2$.}
           \end{array}
         \right.\]

\end{lemma}
 \begin{proof}

Without loss of generality we may assume that
     $(1,y)=(1,1)$. Consider the standard cycle in $Z_{n,s}$. Then $(1,1)=X^s_1$  and $(2, 1+D)=X^s_{\left\lfloor
     \frac{n+3-s+1}{2}\right\rfloor+1}$. The result follows by the observation that the standard cycle in each case is
     isomorphic to $C_{n+3-s}$.
 \end{proof}

The position function for Case 1 is
\begin{eqnarray}
p_1(\alpha_{2i-1})&=\left(1,1+\omega(i-1)\right) \textrm{ and} \nonumber \\
p_1(\alpha_{2i})&=\left(2,1+D+\omega(i-1)\right),
\end{eqnarray}
where $D$ is as defined in Lemma \ref{lem:diameteraway} and
\[\omega=\left\{
           \begin{array}{ll}
             k, & \hbox{if $n=4k+2$ when $k$ is odd, or $n=4k+1$,} \\
             k+1, & \hbox{if $n=4k+2$ when $k$ is even, or $n=4k+3$.}
           \end{array}
         \right.
\]

\begin{lemma} \label{lem:bijection}
The function $p_1:\{\alpha_j\,|\,j=1,...,2n\}\rightarrow V(Z_{n,s})$ is a bijection.
\end{lemma}
\begin{proof}

Suppose $p_1(\alpha_a)=p_1(\alpha_b)$ with $a>b$. Let $i=\left\lfloor \frac{a}{2}\right\rfloor$ and $j=\left\lfloor
\frac{b}{2}\right\rfloor$. As $p_1(\alpha_a)$ and $p_1(\alpha_b)$ have the same first coordinate, $a$ and $b$ have the same
parity. Examining the second coordinates we can conclude that $i \omega \equiv j \omega \mod n$ or $(i -j)\omega
\equiv 0 \mod n$.

By Euclid's algorithm, $k$ is co-prime to $4k+1$ and $k$ is co-prime to $4k+2$ when $k$ is odd. Also, $k+1$ is co-prime
to $4k+2$ when $k$ is even and $k+1$ is co-prime to $4k+3$ for all $k$. Thus in all cases gcd$(n,\omega)=1$. As $n$
divides $(i-j)\omega$, it follows that $n$ divides $(i-j)$. But then $(i-j)\geq n$ and thus $a-b\geq 2n$, so $a>2n$, a
contradiction.
 \end{proof}

Lemma \ref{lem:bijection} establishes that the function $c \circ p_1^{-1}:V(Z_{n,s})\rightarrow \textbf{N}$ assigns each vertex
exactly one label. It remains to show that the labeling satisfies the radio condition. The following lemma simplifies many
of the calculations needed.

\begin{lemma}\label{lem:usefulfacts}
 In all cases considered,
 \begin{itemize}
 \item
 \noindent $\phi(n,s)+\omega \geq diam(Z_{n,s})+1$
and

\item $\phi(n,s)-\omega\geq\left\{
           \begin{array}{ll}
             1, & \hbox{if $n-s$ is even} \\
             2, & \hbox{if $n-s$ is odd.}
           \end{array}
         \right.
$
\end{itemize}

\end{lemma}
\begin{proof}

First, we give the values of $\diam(Z_{n,s})+1$:
\begin{center}
   \begin{tabular}{c|c|c|c}
   {$diam(Z_{n,s})+1$} & {$s=1$} & {$s=2$} & { $s=3$} \\ \hline
   {$r=1$ } & {$2k+2$} & {$2k+2$} & {$2k+1$} \\
  \hline
     $r=2$  & {$2k+3$} & $2k+2$ &  {$2k+2$} \\
     \hline

  {$r=3$} & {$2k+3$} & {$2k+3$} & {$2k+2$} \\
  \hline
     \end{tabular}
     \end{center}
\medskip
In each case, $\diam(Z_{n,s})+1 \leq \phi(n,s)+\omega$:
\medskip
\begin{center}
   \begin{tabular}{c|c|c|c}
   {$\phi(n,s)+\omega$} & {$s=1$} & {$s=2$} & { $s=3$} \\ \hline
   {$r=1$ } & {$2k+2$} & {$2k+2$} & {$2k+1$} \\
  \hline
     $r=2$, $k$ odd  & {$2k+3$} & $2k+2$ &  {$2k+2$} \\
     \hline
 {$r=2$, $k$ even } & {$2k+4$} & {$2k+3$} & {} \\
  \hline
  {$r=3$} & {$2k+3$} & {$2k+4$} & {$2k+3$} \\
  \hline
     \end{tabular}
     \end{center}
\medskip
The last table shows the values of $\phi(n,s)-\omega$ with the entries corresponding to $n+s \equiv 0 \mod 2$
    in bold.
\medskip
\begin{center}
   \begin{tabular}{c|c|c|c}
   {$\phi(n,s)-\omega$} & {$s=1$} & {$s=2$} & { $s=3$} \\ \hline
   {$r=1$ } & {${\mathbf 2}$} & {$2$} & {${\mathbf 1}$} \\
  \hline
     $r=2$, $k$ odd  & {$3$} & ${\mathbf 2}$ &  {$2$} \\
     \hline
 {$r=2$, $k$ even } & {$2$} & {${\mathbf 1}$} & {} \\
  \hline
  {$r=3$} & {${\mathbf 1}$} & {$2$} & {${\mathbf 1}$} \\
  \hline
     \end{tabular}
     \end{center}
\medskip
\end{proof}

\begin{thm}
The function $c \circ p_1^{-1}:V(Z_{n,s})\rightarrow \textbf{N}$ defines a radio labeling on $Z_{n,s}$ for the values of $n$
and $s$ considered in Case 1.
\end{thm}
\begin{proof}
 By Lemma
\ref{lem:vexdis4AP} it is enough to check that all pairs of vertices in the set $\alpha_j,.., \alpha_{j+3}$ satisfy the radio condition. As $d(\alpha_j, \alpha_{j+a})$ depends only on $a$ and on the parity of $j$, it is enough to check all pairs of the form $(\alpha_{2i-1}, \alpha_{2i-1+a})$ and all pairs of the form $(\alpha_{2i}, \alpha_{2i+a})$ for $a \leq 3$ and for some $i$. To simplify the computations, we will check these pairs in the case when $i=1$. For the convenience of the reader, we give the coordinates and the labels of the relevant vertices.

\begin{center}
\begin{tabular}{l|l|l}
 & vertex & label value \\
\hline
$\alpha_1$ & $(1,1)$ & $1$ \\
$\alpha_2$ & $(2,1+D)$ & $2$ \\
$\alpha_3$ & $(1,1+\omega)$ & $1+\phi(n,s)$ \\
$\alpha_4$ & $(2,1+D+\omega)$ & $2+\phi(n,s)$ \\
$\alpha_5$ & $(1,1+2\omega)$ & $1+2\phi(n,s)$ \\
\end{tabular}
\end{center}

{\bf Pair $(\alpha_1, \alpha_2)$:} By Lemma
 \ref{lem:diameteraway}, $d(\alpha_1,\alpha_2)=\diam(Z_{n,s})$. Thus
 $d(\alpha_1,\alpha_2)+c(\alpha_2)-c(\alpha_1)=\diam(Z_{n,s})+1$, as required.

{\bf Pairs $(\alpha_1, \alpha_3)$ and $(\alpha_2, \alpha_4)$:} Note that $\alpha_1$ and $\alpha_3$ lie on the same principal
$n$-cycle and this cycle is tight by Remark \ref{rmk:principletightcycles}. Thus 
\[d(\alpha_1, \alpha_3)+c(\alpha_3)-c(\alpha_1)=\omega+\phi(n,s)\geq \diam
(Z_{n,s})+1.\] 
The last inequality follows by Lemma \ref{lem:usefulfacts}. The relationship between $\alpha_2$ and
$\alpha_4$ is identical.

{\bf Pair $(\alpha_1, \alpha_4)$:} Note that 
\[d(\alpha_1, \alpha_4)\geq d(\alpha_1, \alpha_2)-d(\alpha_2,
\alpha_4)=\diam(Z_{n,s})-\omega.\] Thus 
\[d(\alpha_1, \alpha_4)+c(\alpha_4)-c(\alpha_1)\geq
\diam(Z_{n,s})-\omega+\phi(n,s)+1 \geq \diam(Z_{n,s})+2,\] 
where the last inequality follows by Lemma
\ref{lem:usefulfacts}.

{\bf Pair $(\alpha_2, \alpha_3)$:} By subtracting $\omega$ from the second coordinate, we see $d(\alpha_2,
\alpha_3)=d\left((1,1),(2,1+D-\omega)\right)$.

When considered in the standard cycle, these vertices correspond to $X^1_1$ and $X^1_{2+D-\omega}$ if $s=1$ and to
$X^k_1$ and $X^s_{1+D-\omega}$ if $s=2$ or $3$. As by Remark \ref{rmk:standardtightcycles} the standard cycle in each case is $X_1^s$-tight, we have

\[d(\alpha_2, \alpha_3)=
\left\{
           \begin{array}{ll}
             D-\omega+1=\left\lfloor \frac{n+3}{2}\right\rfloor-\omega, & \hbox{$s=1$,}  \\
             D-\omega=\left\lfloor \frac{n+2}{2}\right\rfloor-\omega, &
             \hbox{$s=2$,}\\
              D-\omega=\left\lfloor
              \frac{n+1}{2}\right\rfloor-\omega, & \hbox{$s=3.$}
           \end{array}
         \right.\]
Thus in all cases $d(\alpha_2, \alpha_3)= \left\lfloor
              \frac{n+3-s+1}{2}\right\rfloor-\omega$, so
\begin{eqnarray}
d(\alpha_2, \alpha_3)&\geq \left\lfloor \frac{n+3-s+1}{2}\right\rfloor-\omega\\
&=
\left\{
           \begin{array}{ll}
          \diam(Z_{n,s})+1-\omega, & \hbox{$n-s$ even,}  \nonumber \\
       \diam(Z_{n,s})-\omega, & \hbox{$n-s$ odd.} \nonumber \\
           \end{array}
         \right.
\end{eqnarray}

By Lemma \ref{lem:usefulfacts},
$\phi(n,s)-\omega \geq 1$ when $n-s$ is even and $\phi(n,s)-\omega \geq 2$ when $n-s$ is odd. Thus
\begin{eqnarray}
d(\alpha_2, \alpha_3)+c(\alpha_3)-c(\alpha_2)&\geq \left\lfloor
\frac{n+3-s+1}{2}\right\rfloor-\omega +(\phi(n,s)-1)\nonumber \\
&\geq \left\{
           \begin{array}{ll}
          \diam(Z_{n,s})+1+1-1, & \hbox{$n-s$ even,} \nonumber  \\
       \diam(Z_{n,s})+2-1, & \hbox{$n-s$ odd.} \nonumber  \\
           \end{array}
         \right.
\end{eqnarray}

{\bf Pair $(\alpha_2, \alpha_5)$:} As $|c(\alpha_5)-c(\alpha_2)|=2\phi(n,s)-1$ and $d(\alpha_2,\alpha_5)\geq1$, it
follows that $|c(\alpha_5)-c(\alpha_2)|+d(\alpha_2,\alpha_5)\geq 2\phi(n,s)$, and so by Remark \ref{rmk:phi},
$|c(\alpha_5)-c(\alpha_2)|+d(\alpha_2,\alpha_5)\geq \diam(Z_{n,s})+1.$

         This establishes that $c \circ p_1^{-1}$ is a radio labeling
         of $Z_{n,s}$.
 \end{proof}

\medskip
\begin{center}
\textbf{Case 2: $n=4k$, $s=1$ or $3$}
\end{center}
In this case the position function is

\begin{eqnarray}
p_2(\alpha_{2i-1})&=(1+l_i, 1+k(i-1)-l_i),\textrm{ and}\nonumber \\
p_2(\alpha_{2i})&=(2+l_i, 1+k(i+1)-l_i),
\end{eqnarray}
where
$l_i=\left\lfloor \frac{i-1}{4}\right\rfloor$.

To simplify notation, we will also denote the value of $l$ associated to a particular vertex $v$ by $l_{(v)}$. Note
that $l_{(\alpha_{2n})}=l_n=l_{4k}=\left\lfloor \frac{4k-1}{4}\right\rfloor\leq k-1$.

\begin{lemma} \label{lem:bijection2}
The function $p_2:\{\alpha_j \,|\, j=1,..,2n\}\rightarrow V(Z_{n,s})$ is a bijection.
\end{lemma}
\begin{proof}
 To show that $p_2$ is a bijection, suppose that
$p_2(\alpha_a)=p_2(\alpha_b)$ and let $i=\left\lfloor \frac{a}{2}\right\rfloor$ and $j=\left\lfloor \frac{b}{2}\right\rfloor$. Suppose first
that $a$ and $b$ are even. Then
\[1+k(i-1)-l_{(a)} \equiv 1+k(j-1)-l_{(b)} \mod n.\] Thus
\[k(i-j)+l_{(b)}-l_{(a)} \equiv 0 \mod 4k,\] so $l_{(b)}-l_{(a)}
\equiv 0 \mod  k$. As $l_{(b)}-l_{(a)}\leq k-1$, this implies that $l_{(b)}=l_{(a)}$. Then we have that
\[k(i-j)+l_{(b)}-l_{(a)}=k(i-j) \equiv 0 \mod 4k,\] and thus
$(i-j)\geq 4$ or $a-b \geq 8$. However, $a-b \geq 8$ implies that $l_{(b)}\neq l_{(a)}$, a contradiction. The argument
when $a$ and $b$ are odd is similar.

Suppose then that $a$ is even and $b$ is odd. Then $1+l_{(b)} \equiv 2+l_{(a)} \mod 2$ shows that $l_{(a)}$ and
$l_{(b)}$ have different parity and in particular $l_{(a)}-l_{(b)} \neq 0$. On the other hand, considering the second
coordinates of $p_2(\alpha_a)$ and $p_2(\alpha_b) \mod k$, we deduce that $1-l_{(a)} \equiv 1-l_{(b)} \mod k$ or $l_{(a)}
-l_{(b)} \equiv 0 \mod k$. As $l_{(a)}-l_{(b)} \neq 0$ it follows that $|l_{(a)} -l_{(b)}|\geq k$, a contradiction.
 \end{proof}

\begin{lemma}
The function $c \circ p_2^{-1}:V(Z_{n,s})\rightarrow \textbf{N}$ defines a radio labeling on $Z_{4k,1}$ and $Z_{4k,3}$.

\end{lemma}
\begin{proof}
 The inequality the function must
satisfy when applied to $Z_{4k,1}$ is $d(u,v)+|c(u)-c(v)|\geq \diam(Z_{4k,1})+1=2k+2$. For $Z_{4k,3}$, the
corresponding inequality is $d(u,v)+|c(u)-c(v)|\geq \diam(Z_{4k,3})+1=2k+1$.

For both $Z_{4k,1}$ and $Z_{4k,3}$, $\phi(n,s)=k+2$, thus $c_{Z_{4k,1}}(u)=c_{Z_{4k,3}}(u)$.

 If $u$ and $v$ are in the same principal cycle, then $d_{Z_{n,3}}(u,v)=
d_{Z_{n,1}}(u,v)$, as principal cycles are always tight. If $u$ and $v$ are on different principal cycles, it is easy
to verify that $d_{Z_{n,3}}(u,v)=d_{Z_{n,1}}(u,v)-1$ by comparing the standard $u$-tight cycles on the two graphs. Thus we can conclude that $d_{Z_{n,3}}(u,v)\geq d_{Z_{n,1}}(u,v)-1$ and so if the radio condition is satisfied by $c_{Z_{n,1}}$, the corresponding radio condition is satisfied by
$c_{Z_{n,3}}$. We will check the radio condition assuming that $s=1$.

As before, it suffices to check that the radio condition holds for all pairs of the form $(\alpha_{2i-1}, \alpha_{2i-1+a})$ and all pairs of the form $(\alpha_{2i}, \alpha_{2i+a})$ for $a \leq 3$. For the convenience of the reader, the relevant values of $p_2$ and $c$ are provided below.

\begin{center}
\begin{tabular}{l|l|l}
 & vertex & label value \\
\hline
$\alpha_{2i-1}$ & $(1+l_{i},1+k(i-1)-l_{i})$ & $1+(i-1)(k+2)$ \\
$\alpha_{2i}$ & $(2+l_{i},1+k(i+1)-l_{i})$ & $2+(i-1)(k+2)$ \\
$\alpha_{2(i+1)-1}$ & $(1+l_{i+1},1+ki-l_{i+1})$ & $1+i(k+2)$ \\
$\alpha_{2(i+1)}$ & $(2+l_{i+1},1+k(i+2)-l_{i+1})$ & $2+i(k+2)$ \\
$\alpha_{2(i+2)-1}$ & $(1+l_{i+2},1+k(i+1)-l_{i+2})$ & $1+(i+1)(k+2)$ \\
\end{tabular}
\end{center}

Note that $l_{(\alpha_{r+a})}-l_{(\alpha_{r})}=0$ or $1$ whenever $a
\leq 3$. As $s=1$, $d\left((x_1, y_1)(x_2,y_2)\right)=|x_2-x_1|+\textrm{min} \{|y_2-y_1|,4k-|y_2-y_1|\}$. The following table has
been generated using this equation.

\begin{center}
\begin{tabular}{c||c|c|c}
{} & {$d(u,v)$} & {$d(u,v)$} & {}\\ {vertex pair} & {$|l_{(u)}-l_{(u)}|=0$} & {$|l_{(u)}-l_{(v)}|=1$}& {$|c(u)-c(v)|$}
\\
 \hline \hline
$(\alpha_{2i-1}, \alpha_{2i})$ & {$1+2k$} & {} & $1$
\\\hline

{$(\alpha_{2i-1}, \alpha_{2(i+1)-1})$} & {$0+k$} &{$1+(k-1)$} & $k+2$\\ \hline

{$(\alpha_{2i-1}, \alpha_{2(i+1)})$} & {$1+k$} &  {$0+(k+1)$}& $k+3$
\\ \hline

{$(\alpha_{2i}, \alpha_{2(i+1)-1})$} & {$1+k$} & {$0+(k+1)$}& $k+1$
\\ \hline

{$(\alpha_{2i}, \alpha_{2(i+1)})$} & {$0+k$} & {$1+(k-1)$} &$k+2$ \\ \hline

{$(\alpha_{2i}, \alpha_{2(i+2)-1})$} & {$1+0$} & {$0+1$} & $2k+3$\\ \hline

\end{tabular}
\end{center}

It is straightforward to verify that in each case, $d(u,v)+|c(u)-c(v)|\geq 2k+2$.
 \end{proof}
\medskip

\begin{center}
\textbf{Case 3: $n=4k$, $s=2$}
\end{center}

The position function for this case is
\begin{eqnarray}
p_3(\alpha_{2i-1})&=(i,1+k(i-1)-l_i),\textrm{ and}\nonumber \\
p_3(\alpha_{2i})&=(i, 1+k(i+1)-l_i),
\end{eqnarray}

where $l_i= \left\lfloor \frac{i-1}{2} \right\rfloor$. Note that $l_{2n}=l_{4k}=\left\lfloor \frac{4k-1}{2}\right\rfloor \leq 2k-1$.

\begin{lemma} \label{lem:bijection3}
The function $p_3:\{\alpha_j\,|\,j=1,..,2n\} \rightarrow V(Z_{n,s})$ is a bijection.
\end{lemma}
\begin{proof}

Suppose that $p_3(\alpha_a)=p_3(\alpha_b)$ and let $i=\left\lfloor \frac{a}{2} \right\rfloor$ and $j=\left\lfloor \frac{b}{2}
\right\rfloor$.  First suppose $a$ and $b$ have the same parity, say even. Then
\[1+k(i+1)-l_{(a)} \equiv 1+k(j+1)-l_{(b)} \mod n.\] Thus
\[k(i-j)+l_{(b)}-l_{(a)} \equiv 0 \mod 4k,\]
so $l_{(b)}-l_{(a)}
\equiv 0\mod k$. As $|l_{(a)}-l_{(b)}|\leq 2k-1$, this implies that $l_{(a)}=l_{(b)}$ or that $l_{(a)}=l_{(b)}+k$.
In the first case it follows that
\[k(i-j)+l_{(b)}-l_{(a)}=k(i-j) \equiv 0 \mod 4k\] and thus
$(i-j)\geq 4$ or $a-b \geq 8$. However, $a-b \geq 8$ implies that $l_{(b)}\neq l_{(a)}$, a contradiction.

If $l_{(a)}=l_{(b)}+k$, it follows that
\[k(i-j)+l_{(b)}-l_{(a)}=k(i-j-1) \equiv 0 \mod 4k,\]
and thus $4$ divides $i-j-1$. We conclude that
$i-j-1$ is even and thus $i-j$ is odd. It follows that $i$ and $j$ have different parities. But in this case
$p_3(\alpha_a)$ and $p_3(\alpha_b)$ have different first coordinates, so $p_3(\alpha_a) \neq p_3(\alpha_b)$. The argument when $a$ and $b$ are odd is similar.

Suppose then that $a$ is even and $b$ is odd.  Considering the second coordinate of $p_3(\alpha_a)-p_3(\alpha_b)$ $ \mod k$ gives that $l_{(b)}-l_{(a)} \equiv 0\mod k$. As $|l_{(a)}-l_{(b)}| \leq 2k-1$, we again conclude that
$l_{(a)}=l_{(b)}$ or $l_{(a)}=l_{(b)}+k$. In the first case, considering the second coordinate of
$p_3(\alpha_a)-p_3(\alpha_b)$ $\mod 2k$, we conclude $k(i-j) \equiv 0 \mod 2k$, so $(i-j)\geq 2$. This however
implies that $l_{(b)}\neq l_{(a)}$, a contradiction. If $l_{(a)}=l_{(b)}+k$, then, should the second coordinate of
$p_3(\alpha_a)-p_3(\alpha_b)$ be congruent to 0 $\mod 4k$, we'd have $2k(i-j-1) \equiv 0 \mod 4k$, so $(i-j-1)$ is even. Again this shows that $p_3(\alpha_a)$ and $p_3(\alpha_b)$ have different first coordinates, so can not be equal.
\end{proof}

\begin{lemma}
The function $c \circ p_3^{-1}:V(Z_{n,s})\rightarrow \textbf{N}$ defines a radio labeling on $Z_{4k,2}$.

\end{lemma}
\begin{proof}
As before it suffices to check all pairs of the form $(\alpha_{2i-1}, \alpha_{2i-1+a})$ and all pairs of the form
$(\alpha_{2i}, \alpha_{2i+a})$ for $a \leq 3$. For the convenience of the reader, the values of $p_3$ for the pairs of vertices we must check are provided below.

\begin{center}
\begin{tabular}{l|l|l}
 & vertex & label value \\
\hline
$\alpha_{2i-1}$ & $(i,1+k(i-1)-l_{i})$ & $1+(i-1)(k+1)$ \\
$\alpha_{2i}$ & $(i,1+k(i+1)-l_{i})$ & $2+(i-1)(k+1)$\\
$\alpha_{2(i+1)-1}$ & $(i+1,1+ki-l_{i+1})$ & $1+i(k+1)$\\
$\alpha_{2(i+1)}$ & $(i+1,1+k(i+2)-l_{i+1})$ & $2+i(k+1)$\\
$\alpha_{2(i+2)-1}$ & $(i+2,1+k(i+1)-l_{i+2})$ & $=1+(i+1)(k+1)$\\
\end{tabular}
\end{center}

We will have to compute distances in $Z_{n,2}$. It is easy to see that $d\left((1,j),(2,j)\right)=1$ and $d\left((i, j),(i,j')\right)=\textrm{min} \{|j-j'|, 4k-|j-j'| \}$. The distance $d\left((1, j)(2,j')\right)$, $j\neq j'$, is somewhat harder to compute. For this purpose we can use the standard cycle in $Z_{4k,2}$ after appropriate renaming of the
vertices. In particular, $d\left((1, j),(2,j')\right)=d\left((1, j-j+1),(2,j'-j+1)\right)=d\left((1, 1),(2,j'-j+1)\right)$. Let $r \equiv j'-j+1$ $\mod 4k$ and $r \in\{1,...,n\}$. Then $d\left((1, j),(2,j')\right)=d\left((1, 1),(2,r)\right)=d_{C_{n+1}}(X^2_1, X^2_r)=\textrm{min} \{r-1,
n+1-(r-1) \}$. Note that in $Z_{n,2}$, $d\left((1, j)(2,j')\right) \neq d\left((1, j')(2,j)\right)$ thus $d(\alpha_s,
\alpha_t)$ depends on the parities of $\left\lfloor \frac{s}{2} \right\rfloor$ and $\left\lfloor \frac{t}{2} \right\rfloor$.

The following table shows the distances and label differences of the relevant pairs computed using the methods described above.

\begin{center}
\begin{tabular}{c||c|c|c}
{vertex pair} & {$d(u,v)$, $i$ even} & {$d(u,v)$, $i$ odd} & {$|c(u)-c(v)|$}\\

 \hline \hline

$(\alpha_{2i-1}, \alpha_{2i})$ & {$2k$} & {$2k$} &  {$1$} \\\hline

{$(\alpha_{2i-1}, \alpha_{2(i+1)-1})$} & {$d\left((1,1),(2,3k+2)\right)$} & {$d\left((1,1)(2,k+1)\right)$} & $k+1$
\\
{}&$=k$&$=k$ &{} \\ \hline

{$(\alpha_{2i-1}, \alpha_{2(i+1)})$} & {$d\left((1,1),(2,k+2)\right)$} & {$d\left((1,1),(2,3k+1)\right)$} & $k+2$\\

{}&$=k+1$&$=k+1$&{} \\ \hline

{$(\alpha_{2i}, \alpha_{2(i+1)-1})$} & {$d\left((1,1),(2,k+2)\right)$} & {$d\left((1,1),(2,3k+1)\right)$} & $k$
\\
{}&$=k+1$&$=k+1$&{} \\

\hline

{$(\alpha_{2i}, \alpha_{2(i+1)})$} & {$d\left((1,1),(2,3k+2)\right)$} & {$d\left((1,1),(2,k+1)\right)$} & $k+1$ \\

{}&$=k$&$=k$&{}\\ \hline

{$(\alpha_{2i}, \alpha_{2(i+2)-1})$} & {$\geq 1$} & {$\geq 1$} & $2k+1$ \\ \hline

\end{tabular}
\end{center}
 \end{proof}

\begin{center}
\textbf{Case 4: $n=4k+2$ when $k$ is even
 and $s=3$}
\end{center}

The position function is
\begin{eqnarray}
p_4(\alpha_{2i-1})&=(l_i, 1+(i-1)k), \textrm{ and} \nonumber \\
p_4(\alpha_{2i})& =(l_i, 2+(i+1)k),
\end{eqnarray}
where
\[l_i=\left\{
           \begin{array}{ll}
             0, & \hbox{$i \leq 2k+1$,}  \\
             1, &
             \hbox{$2k+1 < i \leq 4k+2$.}\\
           \end{array}
         \right.\]

\begin{lemma} \label{lem:bijection4}
The function $p_4: \{\alpha_j \,|\, j=1,...,2n\} \rightarrow V(Z_{n,s})$ is a bijection.
\end{lemma}

\begin{proof}

Suppose $p_4(\alpha_{a})=p_4(\alpha_{b})$. Let $i=\left\lfloor \frac{a}{2} \right\rfloor$ and $j=\left\lfloor \frac{b}{2} \right\rfloor$. If
$a$ and $b$ have the same parity, it follows that $ki \equiv kj \mod (4k+2)$, i.e., $(i-j)k=(4k+2)m$ for some integer
$m$. As $k$ is even, $k=2q$ for some integer $q$. Substituting and simplifying, we obtain the equation
$q(i-j)=m(4q+1)$. As $\gcd(q, 4q+1)=1$, it follows that $q|m$ and thus $m \geq q$, so $(i-j) \geq 4q+1=2k+1$. But in this
case $l_j \neq l_i$, so the first coordinates of $p_4(\alpha_{a})$ and $p_4(\alpha_{b})$ are different.

If $a$ is odd and $b$ is even, it follows that $1+(j+1)k-(i-1)k \equiv 0 \mod 4k+2$. So $1+k(j-i+2) \equiv 0 \mod 4k+2$. As $k$ is even by hypothesis, $1+k(j-i+2)$ is odd, but $4k+2$ is even, a contradiction.
 \end{proof}

\begin{lemma}
The function $c \circ p_4^{-1}:V(Z_{n,s})\rightarrow \textbf{N}$ defines a valid radio labeling on $Z_{4k+2,3}$ when $k$ is even.

\end{lemma}

\begin{proof}
 Since
diam$(Z_{4k+2,3})=2k+1$, we need to show that $d(u,v)+|c(u)-c(v)|\geq 2k+2$ for all pairs $u,v \in Z_{4k+2,3}$. Again
it suffices to check only the pairs $(\alpha_{2i-1}, \alpha_{2i-1+a})$ and the pairs of the form $(\alpha_{2i},
\alpha_{2i+a})$ for $a \leq 3$. Below are given the positions and the labels of these vertices.
\medskip

\begin{center}
\begin{tabular}{l|l|l}
 & vertex & label\\
\hline
$\alpha_{2i-1}$ & $\left(l_i,1+k(i-1)\right)$ & $1+(i-1)(k+2)$ \\
$\alpha_{2i}$ & $\left(l_i,2+k(i+1)\right)$ & $2+(i-1)(k+2)$\\
$\alpha_{2(i+1)-1}$ & $(l_{i+1},1+ki)$ & $1+i(k+2)$\\
$\alpha_{2(i+1)}$ & $\left(l_{i+1},2+k(i+2)\right)$ & $2+i(k+2)$\\
$\alpha_{2(i+2)-1}$ & $\left(l_{i+2},1+k(i+1)\right)$ & $1+(i+1)(k+2)$\\
\end{tabular}
\end{center}

Note that in $Z_{n,3}$, $d\left((x_1,y_1),(x_2,y_2)\right)=\textrm{min} \{|y_2-y_1|, n-|y_1-y_2| \}$ so the first coordinates of
the vertices are irrelevant when computing distances. As $l_i$ only appears in the first coordinates, we do not have to consider the cases of
$l_i=l_{i+1}$ and $l_i \neq l_{i+1}$ separately. Below are given all the relevant distances and label differences. It
is easy to verify that the condition $d(u,v)+|c(u)-c(v)|\geq 2k+2$ is satisfied for all pairs.

\begin{center}
\begin{tabular}{c||c|c}
{vertex pair} & {$d(u,v)$} & {$|c(u)-c(v)|$} \\
 \hline \hline

$(\alpha_{2i-1}, \alpha_{2i})$ & {$2k$+1} & $1$  \\\hline

{$(\alpha_{2i-1}, \alpha_{2(i+1)-1})$} & {$k$}  & {$k+2$} \\ \hline

{$(\alpha_{2i-1}, \alpha_{2(i+1)})$} & {$1+k$} & {$k+3$} \\ \hline

{$(\alpha_{2i}, \alpha_{2(i+1)-1})$} & {$k+1$} & {$k+1$}
\\ \hline

{$(\alpha_{2i}, \alpha_{2(i+1)})$} & {$k$} & {$k+2$} \\ \hline

{$(\alpha_{2i}, \alpha_{2(i+2)-1})$} & {$1$} & {$2k+3$} \\ \hline
\end{tabular}
\end{center}
 \end{proof}

\end{document}